\documentclass[12pt,a4paper,bibliography=totocnumbered]{article}
\usepackage[naustrian,english]{babel}
\usepackage{lmodern} 
\usepackage{amsmath,amstext,amsfonts,amssymb,amsthm, bm}
\usepackage{latexsym}
\usepackage[T1]{fontenc} 
\usepackage[utf8]{inputenc} 
\usepackage{amsfonts}
\usepackage[paper=a4paper, left=3cm, right=1.5cm, top=4cm, bottom=4cm]{geometry} 
\usepackage{titlesec}
\usepackage{float}
\usepackage{graphicx} 
\usepackage{epstopdf} 
\usepackage{enumitem}
\usepackage{hyperref}

\newlist{abbrv}{itemize}{1}
\setlist[abbrv,1]{label=,labelwidth=1in,align=parleft,itemsep=0.1\baselineskip,leftmargin=!}

\newtheorem{thm}{Theorem}
\newtheorem{cor}{Corollary}[section]

\newtheorem{lem}{Lemma}[section]
\newtheorem{defi}{Definition}

\newtheorem{rem}{Remark}
\newcommand{\abs}[1]{\ensuremath{\left\vert#1\right\vert}}

\newcommand{\p}{\mathcal{P}}

\newcommand{\R}{\mathbb{R}}
\newcommand{\N}{\mathbb{N}}
\newcommand{\Z}{\mathbb{Z}}

\title{A survey on M.\@ B. Levin's proofs for the exact lower discrepancy bounds of special sequences and point sets}
\author{Lisa Kaltenb\"{o}ck \and Wolfgang Stockinger \footnote{The authors are supported by the Austrian Science Fund (FWF), Project F5507-N26, which is a part of the Special Research Program “Quasi-Monte Carlo Methods: Theory and Applications”.}}
\date{}
\begin{document}
\maketitle
\begin{abstract}
The goal of this overview article is to give a tangible presentation of the breakthrough works in discrepancy theory \cite{not1,not3} by M.\@ B. Levin. These works provide proofs for the exact lower discrepancy bounds of Halton's sequence and a certain class of $(t,s)$-sequences. Our survey aims at highlighting the major ideas of the proofs and we discuss further implications of the employed methods. Moreover, we derive extensions of Levin's results.  
\end{abstract} 
\section{Introduction and statement of main results}
In \cite{not1} and \cite{not3} M.\@ B. Levin proved optimal lower discrepancy bounds for certain shifted $(t,m,s)$-nets and for the $s$-dimensional Halton sequence. The main ideas of these proofs are also basis for later, even deeper works of Levin on this topic, see \cite{not2,not5}. However, these papers will not be discussed in our survey. In \cite{not1} and \cite{not3} Levin showed the subsequent Theorems 1 and 2, which we will state below in a simplified version. We start with fixing the notation for basic quantities and concepts, which will be needed for the formulation of Levin's results and of our extensions. \\ \\ Let $(\boldsymbol{x}_n)_{n \in \mathbb{N}}$ be an infinite sequence in the $s$-dimensional unit cube $[0,1)^s$, 
\begin{equation*}
\boldsymbol{y} =(y^{(1)}, \ldots, y^{(s)}),
\end{equation*}
and 
\begin{equation*}
[\boldsymbol{0},\boldsymbol{y}) = [0,y^{(1)}) \times \ldots \times  [0,y^{(s)}) \subseteq [0,1)^s.
\end{equation*}
We call $\Delta(\cdot, (\boldsymbol{x}_n)_{n=1}^{N}): [0,1]^s \to \mathbb{R}$, 
\begin{equation*}
\Delta(\boldsymbol{y},(\boldsymbol{x}_n)_{n=1}^{N}) = \sum_{n=1}^{N} (\strut\chi_{[\boldsymbol{0},\boldsymbol{y})}(\boldsymbol{x}_n) - y^{(1)} \cdots y^{(s)} ),
\end{equation*} 
the discrepancy function of the sequence $(\boldsymbol{x}_n)_{n \in \mathbb{N}}$. We define the star-discrepancy of an $N$-point set $(\boldsymbol{x}_n)_{n=1}^{N}$ as
\begin{equation*}
D^{*}((\boldsymbol{x}_n)_{n=1}^N)= \sup_{\boldsymbol{y} \in [0,1)^s} \abs{\frac{1}{N} \Delta(\boldsymbol{y},(\boldsymbol{x}_n)_{n=1}^{N})}.
\end{equation*}
Further, we need the definition of a $(t,m,s)$-net in base $b$ introduced by H.\ Niederreiter \cite{not7} and the so-called $d$-admissibility property of nets. 
\begin{defi}
For integers $b \geq 2$, $s \geq 1$, $m$ and $t$, with $0 \leq t \leq m$, a \textit{$(t,m,s)$-net in base $b$} is defined as a set of points $\mathcal{P} = \lbrace \boldsymbol{x}_0, \ldots, \boldsymbol{x}_{b^m-1} \rbrace$ in $[0,1)^s$, which satisfies the condition that every interval with volume $b^{-m+t}$ of the form $\mathcal{J}=\prod_{i=1}^s \big[\frac{a_i}{b^{d_i}},\frac{a_i+1}{b^{d_i}} \big)$, with $d_i \in \mathbb{N}_0$, $a_i \in \lbrace 0,1, \ldots, b^{d_i} -1 \rbrace$, for $i=1, \ldots, s$, contains exactly $b^t$ points of $\mathcal{P}$. We will call these intervals $\mathcal{J}$ elementary intervals. 
\end{defi}
\begin{defi}
For $x = \sum_{i \geq 1} \frac{x_i}{b^i}$, where $x_i \in \{0,1,...,b-1\}$ and $m \in \N$, the truncation is defined as
\begin{equation*}
[x]_m = \sum_{i=1}^{m}\frac{x_i}{b^i}.
\end{equation*}
For $\bm{x} = (x^{(1)},...,x^{(s)})$ the truncation is defined as $[\bm{x}]_m = ([x^{(1)}]_m,...,[x^{(s)}]_m)$. Moreover, we define $[x]_0 :=0$.
\end{defi}
Keep in mind that for an arbitrary number $x \in \R$, $[x]$ denotes the integer part of $x$.
For the next definition recall the concept of the digital shift. For a point $x = \sum_{i \geq 1} \frac{x_i}{b^i}$ and a shift $\sigma = \sum_{i \geq 1} \frac{\sigma_i}{b^i}$ we have that 	
\begin{equation*}
x \oplus \sigma := \sum_{i \geq 1} \frac{y_i}{b^i}, \qquad \text{ where } \qquad y_i \equiv x_i + \sigma_i \mod{b}
\end{equation*}	
and analogously	
\begin{equation*}
	x \ominus \sigma := \sum_{i \geq 1} \frac{y_i}{b^i}, \qquad \text{ where } \qquad y_i \equiv x_i - \sigma_i \mod{b}.
\end{equation*}	
For $\bm{x} = (x^{(1)},...,x^{(s)})$ and $\bm{\sigma} = (\sigma^{(1)},...,\sigma^{(s)})$ the $b$-adic digitally shifted point is defined by $\bm{x} \oplus \bm{\sigma} = 
(x^{(1)} \oplus \sigma^{(1)},...,x^{(s)} \oplus \sigma^{(s)})$. Analogously we define $\bm{x} \ominus \bm{\sigma}$.
\begin{defi}
For $x = \sum_{i \geq 1} \frac{x_i}{b^i}$, where $x_i = 0$ for $i = 1,...,k$ and $x_{k+1} \neq 0$, the absolute valuation of $x$ is defined as
\begin{equation*}
\|x\|_b = \frac{1}{b^{k+1}}.
\end{equation*}
For $\bm{x} = (x^{(1)},...,x^{(s)})$ the absolute valuation is defined as $\|\bm{x}\|_b := \prod_{j=1}^{s}\|x^{(j)}\|_b$.
\end{defi}
With this definition we can introduce point sets with a special property which is essential for the further considerations of this chapter.
\begin{defi}
For an integer $d$, we say that a point set $\p = \{\bm{x}_0,...,\bm{x}_{b^m-1}\}$ in $[0,1)^s$ is $d$-admissible in base $b$ if
\begin{equation*}
\min_{0 \leq k <n <b^m} \|\bm{x}_n \ominus \bm{x}_k\|_b > \frac{1}{b^{m+d}}.
\end{equation*}
\end{defi}
We remind the definition of the Halton sequence in bases $b_1, \ldots, b_s$, where $s \geq 1$. Throughout this survey all occurring bases $b_1, \ldots, b_s$, are assumed to be pairwise coprime integers.
\begin{defi}
Let $b_1, \ldots, b_s, \ b_i \geq 2$ $(i=1, \ldots, s)$, for some dimension $s \geq 1$, be integers. Then the \textit{$s$-dimensional Halton sequence in bases $b_1, \ldots, b_s$}, denoted by $(H_s(n))_{n \in \mathbb{N}_0}$, is defined as  
\begin{equation*}
H_s(n):= (\phi_{b_1}(n), \ldots, \phi_{b_s}(n)), \quad n=0,1,\ldots, 
\end{equation*} where $\phi_{b_i}$ denotes the radical inverse function in base $b_i$, i.e, the function $\phi_{b_i} : \mathbb{N}_0 \to [0,1)$, defined as
\begin{equation*}
\phi_{b_i}(n):= \sum_{j=0}^{\infty} n_j b_{i}^{-j-1}, 
\end{equation*}
where $n= n_0 + n_1b_i + n_2 b_{i}^2 + \ldots, \text{ with } n_0, n_1, n_2, \ldots \in \lbrace 0,1, \ldots, b_i-1 \rbrace$. 
\end{defi}
It is well known in discrepancy theory that the Halton sequence (requiring that the underlying bases are pairwise coprime) is a low discrepancy sequence, i.e., the star-discrepancy is of order $\mathcal{O}\big(\frac{(\log N)^s}{N}\big)$ (see, e.g., \cite{not4}). Succeeding in showing that the discrepancy of the Halton sequence satisfies $D^{*}((H_s(n))_{n=1}^{N}) \geq c_s \frac{(\log N)^s}{N}$, for infinitely many $N$, with a constant $c_s >0$, would prove that this order is exact. \\ \\
For $(t,m,s)$-nets in base $b$, denoted by $\mathcal{P}$, we know that their discrepancy always satisfies $D^{*}(\mathcal{P}) \leq c_{s,b} b^{t} \frac{(\log N)^{s-1}}{N}$. We will show that the order $\mathcal{O}\big(\frac{(\log N)^{s-1}}{N}\big)$ is exact for certain $(t,m,s)$-nets. 
\\ \\ 
Now, we can state Levin's main results from \cite{not1} and \cite{not3} (in a simplified form).     
\begin{thm}
Let $s\geq 2, d \geq 1, m \geq 9(d+t)(s-1)^2$ and let $(\boldsymbol{x}_n)_{0 \leq n < b^m}$ be a $d$-admissible $(t,m,s)$-net in base $b$. Then, we can provide an explicitly given $\bm{w}$ such that 
\begin{equation*}
b^m D^{*}((\bm{x}_n \oplus \bm{w})_{0 \leq n < b^m}) \geq \frac{(4(d+t)(s-1)^2)^{-s+1}}{b^d} m^{s-1}.
\end{equation*}
In particular, we have 
\begin{equation*}
D^{*}((\bm{x}_n \oplus \bm{w})_{0 \leq n < N}) \geq c_{s,d} \frac{(\log N)^{s-1}}{N}, 
\end{equation*} 
with a constant $c_{s,d} > 0$ and $N=b^m$.
\end{thm}
\begin{thm}
Put $B = b_1 \cdots b_s$, $s \geq 2$ and $m_0= \lfloor 2 B \log_2 B  \rfloor +2$, then the estimate for the star-discrepancy of the Halton sequence
\begin{equation*}
\sup_{1 \leq N \leq 2^{mm_0}} N D^{*}((H_s(n))_{n=1}^{N}) \geq m^s (8 B)^{-1},
\end{equation*} is valid for $m \geq B$.
In particular, there exists some constant $c_s > 0$, such that 
\begin{equation*}
D^{*}((H_s(n))_{n=1}^{N}) \geq c_s \frac{(\log N)^{s}}{N}, \text{ for infinitely many } N  \in \mathbb{N}.
\end{equation*}   
The implied constant $c_s$ also depends on the bases but not on $N$.
\end{thm}  
The aim of this paper is two-fold. \textbf{First}, we will give an easier and simpler access to the ideas of Levin. To this end, we are eager to give a clear and illustrative re-proof of Theorems 1 and 2. We use absolutely the same ideas as Levin, but focus on a clearer presentation. To achieve this goal, we restrict the re-proof of Theorem 1 to the two-dimensional case and carry out the steps in detail. For this case of course, the exact lower discrepancy bound follows (for an arbitrary $\bm{w}$) by the general lower bound for the discrepancy of two-dimensional point sets by W.\ M. Schmidt \cite{not6}. For simplicity we will also restrict ourselves to base $b = 2$. Moreover, we focus on the optimal quality parameter $t = 0$ and for ease of presentation we formulate and prove the result for $m \equiv 0 \mod 4$. We also state the result without the shift and require a certain condition on $\bm{x}_0$ instead. (The ideas for the proof in the general case are the same as in this special version.) This gives Theorem 3:
\begin{thm}
\label{lma:levin1}
Let $(\bm{x}_n)_{0 \leq n < 2^m}$ be a $(0,m,2)$-net in base $2$ with $m \geq 4$, $m \equiv 0 \mod 4$ and $\bm{x}_0 = \bm{\gamma} = (\gamma^{(1)},\gamma^{(2)})$,
\begin{equation*}
\label{ass:1-gamma}
\begin{split}
&\gamma^{(1)} = \frac{1}{2^2} + \frac{1}{2^4}+ \cdots + \frac{1}{2^{m/2}},\\
&\gamma^{(2)} = \frac{1}{2^{m/2+2}} + \frac{1}{2^{m/2+4}} + \cdots + \frac{1}{2^m}.
\end{split}
\end{equation*}
Then it holds for the interval $J_{\bm{\gamma}} = [0,\gamma^{(1)}) \times [0,\gamma^{(2)})$ that
\begin{equation*}
\frac{1}{N} \Delta (\bm{\gamma},(\bm{x}_n)_{0 \leq n < 2^m})
\leq -\frac{1}{4} \frac{1}{2^{m+2}} m,
\end{equation*}
and consequently 
\begin{equation*}
D^{*}((\bm{x}_n)_{0 \leq n < N}) \geq \frac{1}{16 \log 2} \frac{\log N}{N},
\end{equation*}
with $N=2^m$.
\end{thm} 
The \textbf{second aim} is to give a - in a certain sense - quantitative extension of Theorems 1 and 2. We will show: 
\begin{thm}
\label{thm:epsilon-umgebung-allgemein}
Let $m \geq 2s^s(s-1)^s$. Then, there is a set $\Gamma \subseteq [0,1)^s$, $s\geq 2$, with the following properties: 
\begin{itemize}
\item $\text{For all } \boldsymbol{x} \in [0,1)^s$ there exists a $\boldsymbol{\gamma} \in \Gamma$ with 
\begin{equation*}
\| \boldsymbol{x} - \boldsymbol{\gamma} \| < b\sqrt{s} \frac{1}{b^{ \frac{m}{2(s-1)s}}}.
\end{equation*}
Here, $\| \cdot \|$ denotes the euclidean norm. \\
\item If $ \mathcal{P} = \lbrace \bm{x}_0, \ldots, \bm{x}_{b^m -1} \rbrace$ is a $(0,m,s)$-net in base b, and if $\boldsymbol{x}_i \in \Gamma$ for some $i \in \lbrace 0, \ldots, b^m-1 \rbrace$, then, with $N=b^m$, 
\begin{equation*}
D^{*}(\mathcal{P}) \geq \frac{(b-1)^s(2s-3)^{s-1}}{b^s(4s^2(s-1)^2 \log b)^{s-1}} \frac{(\log N)^{s-1}}{N}.
\end{equation*} 
\end{itemize}
\end{thm}
\begin{thm}
There are constants $c_1$ and $c_2 > 0$, such that for infinitely many $N$ there exists a set $\Lambda_N \subseteq [0,1)^2$ with the following properties: 
\begin{itemize}
\item We have $\lambda_2 (\Lambda_N) \geq c_1$, where $\lambda_2$ denotes the 2-dimensional Lebesgue measure.
\item For all $\boldsymbol{x} \in \Lambda_N$ there exists a $\boldsymbol{y} \in [0,1)^2$ with $\| \boldsymbol{x} - \boldsymbol{y} \| < \sqrt{8} \frac{1}{N^{\frac{1}{14}}}$ and 
\begin{equation*}
\abs{\Delta(\boldsymbol{y}, (H_2(n))_{n=1}^N)} \geq c_2 (\log N)^2.
\end{equation*} 
\end{itemize}
\end{thm}   
\begin{rem}
An analogous result can be obtained for arbitrary dimensions. For sake of simplicity our considerations will be restricted to the two-dimensional case. The basic ideas become better visible in this case and can be adopted to higher dimensions in a straightforward manner.   
\end{rem}
The remainder of this paper is organised as follows: In Chapter 2, we will discuss the $d$-admissibility property in more detail. Of course, the proof of Theorem 3 will be the major part of this chapter. We relax some of the conditions of Theorem 3 in Chapter 3 and derive a more general result (Theorem 4). In Chapter 4, we will prove Theorem 2 in detail. Chapter 5 will be solely dedicated to the proof of Theorem 5. 
\section{Remarks on admissibility of nets and Re-proof of Theorem 3}
Before stating the proof of Theorem 3, we discuss the $d$-admissibility property for $(0,m,s)$-nets, since in this theorem we restrict ourselves to the quality parameter $t=0$. 
\begin{lem}
A point set $\p = \{\bm{x}_0,...,\bm{x}_{b^m-1}\}$ in $[0,1)^s$ is $s$-admissible if and only if $\mathcal{P}$ is a $(0,m,s)$-net in base $b$. Moreover, $\mathcal{P}$ cannot be $d$-admissible for $d <s$.
\begin{proof}
Let $\mathcal{P}$ be a $(0,m,s)$-net in base $b$. First, we show that 
\begin{equation*}
\frac{1}{b^{m+s-1}} \geq \min_{0 \leq k <n <b^m} \|\bm{x}_n \ominus \bm{x}_k\|_b,
\end{equation*}
by taking special elementary intervals into account. Since $\p$ is a $(0,m,s)$-net, we know by definition that every elementary interval of order $m$ in base $b$, i.e., every elementary interval with volume $\frac{1}{b^{m}}$, contains exactly one point of $\p$. Therefore, this is also true for intervals of the form		
\begin{equation*}
\left[\frac{k}{b^{m}},\frac{k+1}{b^{m}}\right) \times [0,1)^{s-1}, \qquad k \in \{0,...,b^{m}-1\}.
\end{equation*}		
Now let $\bm{x} = (x^{(1)}, \ldots, x^{(s)})$ be the unique point of $\p$ for which it holds that $x^{(1)} \in \left[0,\frac{1}{b^{m}}\right)$.
Moreover, let $\bm{y} = (y^{(1)}, \ldots, y^{(s)})$ be the point of $\p$ such that $y^{(1)} \in \left[\frac{b-1}{b^{m}},\frac{b}{b^{m}}\right)$. This is equivalent to		
\begin{align*}
0 &\leq x^{(1)} < \frac{1}{b^m} \\
\frac{b-1}{b^m} &\leq y^{(1)} < \frac{1}{b^{m-1}}.
\end{align*}	
Therefore, we know that $x^{(1)}$ and $y^{(1)}$ can be written as	
\begin{align*}
x^{(1)} = \frac{\alpha_{1}}{b^{m+1}} + \frac{\alpha_{2}}{b^{m+2}} +\cdots, \\
y^{(1)} = \frac{b-1}{b^{m}} + \frac{\beta_{1}}{b^{m+1}} + \frac{\beta_{2}}{b^{m+2}} + \cdots,
\end{align*}		
where $\alpha_i, \beta_i \in \{0,1,...,b-1\}$ for $i \geq 1$. Thus, $\|y^{(1)} \ominus x^{(1)}\|_b = \frac{1}{b^{m}}$. Moreover, for $x^{(i)}$ and $y^{(i)}$, $i=2, \ldots, s$, it holds that $\|y^{(i)} \ominus x^{(i)}\|_b \leq \frac{1}{b}$. Therefore, it follows, that
\begin{equation*}
\|\bm{y} \ominus \bm{x} \|_b \leq \frac{1}{b^{m+s-1}}.
\end{equation*}
If we can prove that $\min_{0 \leq k <n <b^m} \|\bm{x}_n \ominus \bm{x}_k\|_b > \frac{1}{b^{m+s}}$, 
then the first implication of the assertion immediately follows. Suppose that there exist points $\bm{x} = (x^{(1)}, \ldots, x^{(s)}), \bm{x} \in \p$ and $\bm{y} = (y^{(1)}, \ldots, y^{(s)}), \bm{y} \in \p$ such that $\|\bm{y} \ominus \bm{x}\|_b \leq \frac{1}{b^{m+s}}$. Then, there exist integers $l^{(1)}, \ldots, l^{(s-1)}$ such that
\begin{equation*}
\| y^{(i)} \ominus x^{(i)} \|_b \leq \frac{1}{b^{l^{(i)}}}, \text{ for } i=1, \ldots, s-1, 
\end{equation*} 
and 
\begin{equation*}
\| y^{(s)} \ominus x^{(s)} \|_b \leq \frac{1}{b^{m+s-l^{(1)}- \ldots - l^{(s-1)}}}.
\end{equation*}
This implies that the first $l^{(i)}-1$ digits of the $b$-adic expansion of $x^{(i)}$ and $y^{(i)}$, $i=1, \ldots, s-1$ are identical. Also, the first $m+s-l^{(1)}- \ldots - l^{(s-1)}-1$ digits of the $b$-adic expansion of $x^{(s)}$ and $y^{(s)}$ are identical. Consequently, $\bm{x}$ and $\bm{y}$ are contained in an elementary interval of volume $\frac{1}{b^m}$. This contradicts our assumption that $\p$ is a $(0,m,s)$-net. \\ \\
Let now $\mathcal{P}$ be an arbitrary $b^m$-point set in $[0,1)^s$ which is not a $(0,m,s)$-net. Then there exists an elementary interval $\mathcal{J}_1 \subseteq [0,1)^s$ of volume $1/b^m$ which contains no point of $\mathcal{P}$ or at least two points of $\mathcal{P}$. In the second case it immediately follows (by the same considerations as above) that $\mathcal{P}$ is not $s$-admissible. Consider now the first case: We can partition $[0,1)^s$ into $b^m$ elementary intervals $\mathcal{J}_i$ of the same shape as $\mathcal{J}_1$. Since $\mathcal{J}_1$ contains no point of $\mathcal{P}$ there exists at least one $i$ such that $\mathcal{J}_i$ contains at least two points, and this again contradicts the $s$-admissibility.    
\end{proof} 
\end{lem}
\begin{rem}
Note, that it might happen that a $(1,m,s)$-net in base $b$ is non-admissible for any integer $d$. To see this, just take $b$ copies of a $(0,m-1,s)$-net in base $b$. This gives an example of a $(1,m,s)$-net in base $b$ which is not $d$-admissible for any $d \in \mathbb{N}$.
\end{rem}
These preliminary considerations put us in the position to prove Theorem 3. In Chapter 3 we give the proof for a more general result in the general case. Note, that for $(t,m,s)$-nets with nonzero quality parameter the $d$-admissibility condition has to be required additionally. The idea underlying the proof of the theorem in the general case is exactly the same. \\ \\
\textbf{Proof of Theorem 3:}\\
Note that by Lemma 2.1 $(\bm{x}_n)_{0 \leq n < 2^m}$ is 2-admissible. To begin with, we want to find a suitable partition of the interval $J_{\bm{\gamma}}$.
Let therefore $\bm{r} = (r_1,r_2) \in \N^2$. For
\begin{equation*}
r_1 = 2j_1 \qquad \text{and} \qquad r_2 = m/2 + 2j_2
\end{equation*}
with $j_1, j_2 \in \{1,...,m/4\}$ it holds that
\begin{align*}
\gamma^{(1)} = \sum_{r_1} \frac{1}{2^{r_1}} \qquad \text{and} \qquad 
\gamma^{(2)} = \sum_{r_2} \frac{1}{2^{r_2}}. 
\end{align*}
Now define the set $A$ which contains all combinations of the indices $r_1$ and $r_2$, i.e.,
\begin{equation*}
A = \{(r_1,r_2)|\; r_1 = 2j_1, \; r_2 = m/2 + 2j_2,\; j_1, j_2 \in \{1,...,m/4\} \}.
\end{equation*}
The partition of $J_{\bm{\gamma}}$ is then given by
\begin{equation*}
J_{\bm{r},\bm{\gamma}} =\left[ [\gamma^{(1)}]_{r_1-1},[\gamma^{(1)}]_{r_1-1} + \frac{1}{2^{r_1}}\right) \times \left[ [\gamma^{(2)}]_{r_2-1},[\gamma^{(2)}]_{r_2-1} + \frac{1}{2^{r_2}}\right),
\end{equation*}
for $(r_1,r_2) \in A$.
Furthermore, let
\begin{align*}
A_1 &= \{\bm{r} \in A|\; r_1 + r_2 \leq m \}, \\
A_2 &= \{\bm{r} \in A|\; r_1 + r_2 = m+1 \}, \\
A_3 &= \{\bm{r} \in A|\; r_1 + r_2 \geq m+2 \},
\end{align*}
such that $A  = A_1 \cup A_2 \cup A_3$.
The intervals $J_{\bm{r},\bm{\gamma}}$ are elementary intervals in base $2$ with volume $\frac{1}{2^{r_1+r_2}}$, i.e., of order $r_1 + r_2$. Moreover, all $J_{\bm{r},\bm{\gamma}}$ are disjoint and therefore, we obtain with 
\begin{equation*}
\mathcal{A}(\bm{r}) := \sum_{n=0}^{2^m-1} \chi_{\strut J_{\bm{r},\bm{\gamma}}}(\boldsymbol{x}_n)
\end{equation*}		
\begin{align*}
\frac{1}{N} \Delta (\bm{\gamma},(\bm{x}_n)_{0 \leq n < 2^m})
&= \sum_{\bm{r} \in A} \left(\frac{\mathcal{A}(\bm{r})}{2^m} - \lambda_2(J_{\bm{r},\bm{\gamma}}) \right)\\
&= \sum_{\bm{r} \in A_1} \left(\frac{\mathcal{A}(\bm{r})}{2^m} - \lambda_2(J_{\bm{r},\bm{\gamma}}) \right)\\ 
& \quad + \sum_{\bm{r} \in A_2} \left(\frac{\mathcal{A}(\bm{r})}{2^m} - \lambda_2(J_{\bm{r},\bm{\gamma}}) \right)\\ 
& \quad + \sum_{\bm{r} \in A_3} \left( \frac{\mathcal{A}(\bm{r})}{2^m} - \lambda_2(J_{\bm{r},\bm{\gamma}}) \right)\\
&=: \Delta_1(\bm{\gamma}) + \Delta_2(\bm{\gamma}) + \Delta_3(\bm{\gamma}).
\end{align*}	
\textit{Consider $\Delta_1$}. Since $(\bm{x}_n)_{0 \leq n < 2^m}$ is a $(0,m,2)$-net, it is fair with respect to all elementary intervals of order $\leq m$. For $\bm{r} \in A_1$ it holds that $r_1 + r_2 \leq m$ and therefore
\begin{equation*}
\Delta_1(\bm{\gamma}) = \sum_{\bm{r} \in A_1} \frac{\mathcal{A}(\bm{r})}{2^m} - \lambda_2(J_{\bm{r},\bm{\gamma}}) = 0.
\end{equation*}	
\textit{Consider $\Delta_2$}. From the condition that $\bm{r} \in A_2 \subseteq A$ we get that 
\begin{equation*}
r_1 = 2j_1 \qquad \text{and} \qquad r_2 = m/2 + 2j_2,
\end{equation*}
where $j_1, j_2 \in \{1,...,m/4\}$. It follows that 		
\begin{equation*}
r_1 + r_2 = m + 2(j_1 + j_2 - m/4).
\end{equation*}	
Since $j_1 + j_2 - m/4 \in \Z$ we know that $2(j_1 + j_2 - m/4) \neq 1$ which is a contradiction to the assumption that $r_1 + r_2 = m+1$ for all $\bm{r} \in A_2$. Therefore, $A_2 = \emptyset$ and $\Delta_2 = 0$. \\ \\
\textit{Consider $\Delta_3$}. As a first step we want to show that $J_{\bm{r},\bm{\gamma}}$ with $r_1 + r_2 \geq m +2$ cannot contain any point of $(\bm{x}_n)_{0 \leq n < 2^m}$ and we will do that by deriving a contradiction. \\	
Suppose there exists $\bm{x}_k \in J_{\bm{r},\bm{\gamma}}$ for some $k < 2^m$ and some $\bm{r} \in A_3$. Then we know for the first coordinate	
\begin{equation*}
[\gamma^{(1)}]_{r_1-1} \leq x_k^{(1)} < [\gamma^{(1)}]_{r_1-1} + \frac{1}{2^{r_1}}
\end{equation*}	
which is equivalent to	
\begin{equation*}
\frac{1}{2^2} +\frac{1}{2^4} + \cdots + \frac{1}{2^{r_1-2}}
\leq \frac{x_{k,1}^{(1)}}{2} + \cdots + \frac{x_{k,r_1-1}^{(1)}}{2^{r_1-1}} + \frac{x_{k,r_1}^{(1)}}{2^{r_1}} + \cdots
< \frac{1}{2^2} + \frac{1}{2^4} + \cdots + \frac{1}{2^{r_1-2}} + \frac{1}{2^{r_1}}.
\end{equation*}	
Therefore, it has to hold that $x_{k,2}^{(1)} = x_{k,4}^{(1)} = ... = x_{k,r_1-2}^{(1)} = 1$ and $x_{k,1}^{(1)} = x_{k,3}^{(1)} = ... = x_{k,r_1-1}^{(1)} = 0$. An analogous procedure can be done for the second coordinate. Hence,	
\begin{equation}
\label{eq:1-x_in_J}
[\gamma^{(1)}]_{r_1-1} = [x_k^{(1)}]_{r_1-1} \qquad \text{ and } \qquad
[\gamma^{(2)}]_{r_2-1} = [x_k^{(2)}]_{r_2-1}.
\end{equation}	
Combining (\ref{eq:1-x_in_J}) and the assumption that $\bm{x}_0 = \bm{\gamma}$ leads to	
\begin{equation*}
[(\bm{x}_k \ominus \bm{x}_0)^{(1)}]_{r_1-1} = 0 \qquad \text{and} \qquad
[(\bm{x}_k \ominus \bm{x}_0)^{(2)}]_{r_2-1} = 0.
\end{equation*}	
Thus, we get $\|x_k^{(i)} \ominus x_0^{(i)} \|_2 \leq \frac{1}{2^{r_i}}$.	
Since $\bm{r} \in A_3$, i.e., $r_1 + r_2 \geq m+2$, it follows that	
\begin{equation*}
\|\bm{x}_k \ominus \bm{x}_0 \|_2 \leq \frac{1}{2^{r_1+r_2}} \leq \frac{1}{2^{m+2}}.
\end{equation*}	
This is a contradiction to the assumption that $(\bm{x}_n)_{0 \leq n < 2^m}$ is a 2-admissible $(0,m,2)$-net in base $2$. Hence, $\mathcal{A}(\bm{r}) = 0$ for all $\bm{r} \in A_3$ and	
\begin{align*}
\Delta_3(\bm{\gamma}) &= \sum_{\bm{r} \in A_3}
\left(\frac{\mathcal{A}(\bm{r})}{2^m} - \lambda_2(J_{\bm{r},\bm{\gamma}}) \right) \\
&= - \sum_{\bm{r} \in A_3} \frac{1}{2^{r_1 + r_2}} \\
&\leq - \sum_{\substack{\bm{r} \in A_3 \\ r_1 + r_2 = m+2}} \frac{1}{2^{m+2}} \\
&= - |A_4|\frac{1}{2^{m+2}}
\end{align*}		
with 		
\begin{align*}
A_4 = \{\bm{r} \in A_3|\;r_1+r_2 = m+2\}.
\end{align*}	
It is easy to see that
\begin{equation*}
|A_4| = \frac{m}{4}
\end{equation*}	
for $m \geq 4$ and $m \equiv 0 \mod 4$, and so we finally get
\begin{align*}
\frac{1}{N}\Delta ({\bm{\gamma}},(\bm{x}_n)_{0 \leq n < 2^m})
&= \Delta_3(\bm{\gamma})\\
&\leq - \frac{1}{2^{m+2}}|A_4| \\
&= -\frac{1}{4} \frac{1}{2^{m+2}} m.
\end{align*}
\qed
\section{Proof of Theorem 4}
The first aim of this section is to focus on the assumption of Theorem 3 that there exists a point $\bm{x}_0 \in \p$ such that $\bm{x}_0 = \bm{\gamma}$ (of course the condition $\bm{x}_0 = \bm{\gamma}$ can be replaced by $\bm{x}_n = \bm{\gamma}$ for any $n \in \lbrace 0, \ldots, 2^m -1 \rbrace$). This restriction on the point set is weakened by showing that there are many possible choices for $\bm{\gamma}$ such that the proof of Theorem 3 can still be performed in an analogous way.
In fact, it turns out that $\bm{\gamma}$ only has to fulfill some simple properties as the following lemma shows:
\begin{lem}
\label{lma:levin5}
Let $(\bm{x}_n)_{0 \leq n < b^m}$ be a $(0,m,s)$-net in base $b$. Let $\bm{x}_0 \in \prod_{j=1}^{s}[\gamma^{(j)}, \gamma^{(j)}+\frac{1}{b^{\max (R_j)}})$, where
\begin{equation*}
\gamma^{(j)} = \sum_{r \in R_j} \frac{a_{r}^{(j)}}{b^{r}},
\end{equation*}
$a_{r}^{(j)} \in \{1,2,...,b-1\}$ and $R_j \subseteq \{1,2,...,m\}$ for $j = 1,...,s$. Here the $R_j$ are arbitrary, but for $\bm{r} = (r_1, r_2, ...,r_s) \in R_1 \times R_2 \times ... \times R_s$, the following constraints need to be satisfied:
\begin{itemize}
\item $|\{\bm{r} | \; m+1 \leq \sum_{j = 1}^{s} r_j <  m + s \}| \leq \frac{m^{s-1}}{\delta}$,
\item $|\{\bm{r} | \; \sum_{j = 1}^{s} r_j = m + \alpha \} | \geq \frac{m^{s-1}}{\beta}$,
\end{itemize}
for some constant $\beta > 0$, some integer $\alpha \geq s$ and for $\delta > \frac{b^\alpha (b^{s-1}-1)\beta}{b^{s-1}}$.
Then, it holds for the interval $J_{\bm{\gamma}} = \prod_{j=1}^{s} [0,\gamma^{(j)})$ that
\begin{equation*}
\frac{1}{N} \Delta (\bm{\gamma}, (\bm{x}_n)_{0 \leq n < b^m})
\leq - \frac{m^{s-1}}{b^m} \left(- \frac{(b-1)^s}{\delta} \frac{b^{s-1}-1}{b^{s-1}} + \frac{(b-1)^s}{\beta} \frac{1}{b^\alpha}\right),
\end{equation*}
where $\left(- \frac{(b-1)^s}{\delta} \frac{b^{s-1}-1}{b^{s-1}} + \frac{(b-1)^s}{\beta} \frac{1}{b^\alpha}\right) > 0$.
\end{lem}
\begin{proof}
Let $A = \{\bm{r}| \; r_j \in R_j, \; j = 1,...,s\}$ be the set of indices which can be split into three disjoint subsets
\begin{align*}
A_1 &= \{\bm{r} \in A|\; \sum_{j=1}^{s} r_j \leq m \}, \\
A_2 &= \{\bm{r} \in A|\;  m+1 \leq \sum_{j=1}^{s} r_j < m+s \}, \\
A_3 &= \{\bm{r} \in A|\; \sum_{j=1}^{s} r_j \geq m+s \}.
\end{align*}
Further let 
\begin{equation*}
A_4= \lbrace \bm{r} \in A| \; \sum_{j=1}^s r_j = m + \alpha  \rbrace.
\end{equation*}
A partition of the interval $J_{\bm{\gamma}}$ is given by the subintervals 
\begin{equation*}
J_{\bm{r},\bm{\gamma},\bm{g}} = \prod_{j=1}^{s} \left[[\gamma^{(j)}]_{r_j-1}+\frac{g_j}{b^{r_j}},[\gamma^{(j)}]_{r_j-1}+\frac{g_j+1}{b^{r_j}}\right)
\end{equation*}
where $\bm{g} = (g_1,...,g_s)$ with $g_j \in \{0,1,...,a_{r_j}-1\}$. The intervals $J_{\bm{r},\bm{\gamma},\bm{g}}$ are disjoint elementary intervals of order $\sum_{j=1}^s r_j$ in base $b$. We define 
\begin{equation*}
\mathcal{A}({\bm{r},\bm{g}}) := \sum_{n=0}^{b^m-1} \chi_{\strut J_{\bm{r},\bm{\gamma},\bm{g}}}(\boldsymbol{x}_n).
\end{equation*}
Then, it is possible to split the estimation of the discrepancy function into three parts corresponding to the sets $A_1, A_2$ and $A_3$,
\begin{align*}
\frac{1}{N} \Delta (\bm{\gamma}, (\bm{x}_n)_{0 \leq n < b^m})
&= \sum_{\substack{\bm{r} \in A_1\\ \bm{g}}} \left(\frac{\mathcal{A}(\bm{r},\bm{g})}{b^m} - \lambda_s(J_{\bm{r},\bm{\gamma},\bm{g}}) \right)\\ 
& \quad + \sum_{\substack{\bm{r} \in A_2\\ \bm{g}}} \left(\frac{\mathcal{A}(\bm{r},\bm{g})}{b^m} - \lambda_s(J_{\bm{r},\bm{\gamma},\bm{g}}) \right)\\ 
& \quad + \sum_{\substack{\bm{r} \in A_3\\ \bm{g}}} \left( \frac{\mathcal{A}(\bm{r},\bm{g})}{b^m} - \lambda_s(J_{\bm{r},\bm{\gamma},\bm{g}}) \right)\\
&= \Delta_1 + \Delta_2 + \Delta_3.
\end{align*}
It follows by the net property and the fact that $J_{\bm{r},\bm{\gamma},\bm{g}}$ are elementary intervals that 
\begin{equation*}
\Delta_1 = \sum_{\substack{\bm{r} \in A_1\\ \bm{g}}} \left(\frac{\mathcal{A}({\bm{r},\bm{g}})}{b^m}-\lambda_s(J_{\bm{r},\bm{\gamma},\bm{g}})\right) = 0.
\end{equation*} 
Since $J_{\bm{r},\bm{\gamma},\bm{g}}$, $\bm{r} \in A_2$, are elementary intervals of order greater or equal to $m+1$, they either contain one point of the $(0,m,s)$-net or they are empty. Let us consider these two cases:
\begin{enumerate}
\item $\exists \; \bm{x}_k \in J_{\bm{r},\bm{\gamma},\bm{g}}$. Then it holds that
\begin{equation*}
\frac{1}{b^m}-\frac{1}{b^{m+1}} \leq \frac{\mathcal{A}({\bm{r},\bm{g}})}{b^m} - \lambda_s(J_{\bm{r},\bm{\gamma},\bm{g}})
= \frac{1}{b^m} - \frac{1}{b^{\sum_{j=1}^s r_j}} \leq \frac{1}{b^m}-\frac{1}{b^{m+s-1}}.
\end{equation*}
\item $\nexists \; \bm{x}_k \in J_{\bm{r},\bm{\gamma},\bm{g}}$. In this case it holds that
\begin{equation*}
-\frac{1}{b^{m+1}} \leq \frac{\mathcal{A}({\bm{r},\bm{g}})}{b^m} - \lambda_s(J_{\bm{r},\bm{\gamma},\bm{g}})
= -\frac{1}{b^{\sum_{j=1}^s r_j}} \leq -\frac{1}{b^{m+s-1}}.
\end{equation*}
\end{enumerate}
Then, by the assumptions on $A_2$ we obtain the estimate
\begin{equation*}
-\frac{1}{b^{m+1}} \frac{m^{s-1}}{\delta} (b-1)^s \leq \Delta_2 \leq \left(\frac{1}{b^{m}}-\frac{1}{b^{m+s-1}}\right) \frac{m^{s-1}}{\delta}(b-1)^s.
\end{equation*}
Now, consider $\Delta_3$. The first step is again to show that $J_{\bm{r},\bm{\gamma},\bm{g}}$ with $\bm{r} \in A_3$ and for all associated $\bm{g}$, cannot contain any point of a $(0,m,s)$-net which has an element $\bm{x}_0 \in \prod_{j=1}^{s}[\gamma^{(j)}, \gamma^{(j)}+\frac{1}{b^{\max (R_j)}})$. The condition that $\bm{x}_0$ is contained in this set, is equivalent to
\begin{equation}
[\gamma^{(j)}]_{r_j} = [x_0^{(j)}]_{r_j}, \qquad \text{ for } j = 1,...,s.
\end{equation}
Suppose there exists $\bm{x}_k \in J_{\bm{r},\bm{\gamma},\bm{g}}$ for some $k < b^m$, some $\bm{r} \in A_3$ and some $\bm{g}$. It then follows that 
\begin{equation*}
[\gamma^{(j)}]_{r_j-1} = [x_k^{(j)}]_{r_j-1}, \qquad \text{for } j = 1,...,s.
\end{equation*}
Therefore,
\begin{equation*}
\|\bm{x}_k \ominus \bm{x}_0 \|_b \leq \frac{1}{b^{\sum_{j=1}^s r_j}} \leq \frac{1}{b^{m+s}}.
\end{equation*}
This is a contradiction to the assumption that $\bm{x}_k$ and $\bm{x}_0$ are elements of a $(0,m,s)$-net in base $b$ because from Lemma 2.1 we know that $\min_{\bm{x}, \bm{y} \in \p} \|\bm{x} \ominus \bm{y}\|_b = \frac{1}{b^{m+s-1}}$.
Hence, all $J_{\bm{r},\bm{\gamma},\bm{g}}$ where $\bm{r} \in A_3$ are empty. Using the fact that $|A_4| \geq \frac{m^{s-1}}{\beta}$, we then get 
\begin{align*}
\Delta_3 &= \sum_{\substack{\bm{r} \in A_3 \\ \bm{g}}}
\left(\frac{\mathcal{A}({\bm{r},\bm{g}})}{b^m} - \lambda_s(J_{\bm{r},\bm{\gamma},\bm{g}})\right) \\
&= - \sum_{\substack{\bm{r} \in A_3\\ \bm{g}}} \frac{1}{b^{\sum_{j=1}^{s} r_j}} \\
&\leq - \sum_{\substack{\bm{r} \in A_4 \\ \bm{g}}}  \frac{1}{b^{m+\alpha}}\\
&\leq  - \frac{m^{s-1}}{\beta}(b-1)^s  \frac{1}{b^{m+\alpha}}.
\end{align*}
Finally, we get the estimate	
\begin{align*}
\frac{1}{N}\Delta (\bm{\gamma}, (\bm{x}_n)_{0 \leq n < b^m}) &= \Delta_1 + \Delta_2 + \Delta_3 \\
&\leq \left(\frac{1}{b^{m}}-\frac{1}{b^{m+s-1}}\right) \frac{m^{s-1}}{\delta}(b-1)^{s} - \frac{m^{s-1}}{\beta}(b-1)^{s}  \frac{1}{b^{m+\alpha}} \\
&= - \frac{m^{s-1}}{b^m} \left(- \frac{(b-1)^s}{\delta} \frac{b^{s-1}-1}{b^{s-1}} + \frac{(b-1)^s}{\beta} \frac{1}{b^\alpha}\right) < 0
\end{align*}	
for $\delta > \frac{b^\alpha (b^{s-1}-1)\beta}{b^{s-1}}$.
\end{proof}
Subsequently, we now derive Theorem 4, which in some sense describes how dense possible choices of $\bm{\gamma}$ are in $[0,1)^s$. \\ \\
\textbf{Proof of Theorem 4:}\\
Let $\Gamma$ be defined as the set, which contains all points of the form $\bm{\gamma} = (\sum_{r_1} \frac{1}{b^{r_1}},..., \sum_{r_s} \frac{1}{b^{r_s}})$,
where $r_i \in R_i \subseteq \{1,2,...,m\}$ for $i = 1,...,s$ and the sets $R_i$ fulfill the following conditions:
\begin{itemize}
\item $|\{(r_1,...,r_s) | \; m+1 \leq \sum_{i = 1}^{s} r_i <  m + s \}| = 0$,
\item $|\{(r_1,...,r_s) | \; \sum_{i = 1}^{s} r_i = m + s \}| \geq \frac{m^{s-1} (2s-3)^{s-1}}{(4s^2(s-1)^2)^{s-1}}$.
\end{itemize}
Consider now the $b$-adic digit expansion of some $\bm{x} = (x^{(1)}, ..., x^{(s)}) \in [0,1)^s$,
\begin{equation*}
x^{(i)} = \sum_{s_i \in S_i} \frac{a_{s_i}}{b^{s_i}},
\end{equation*}
where $S_i \subseteq \mathbb{N}$ is the set of indices for which we have $a_{s_i} \in \{1,2,...,b-1\}$ for $i = 1,...,s$. Now we have to construct a point $\bm{\gamma}$ with the following properties:
\begin{equation}
\label{eq:gamma_first}
\| \boldsymbol{x} - \boldsymbol{\gamma} \| < b\sqrt{s} \frac{1}{b^{ \frac{m}{2(s-1)s}}},
\end{equation}
\begin{equation}
\label{eq:gamma_second}
\bm{\gamma} \in \Gamma, \text{ where } \bm{\Gamma} \text{ is defined as above}. 
\end{equation}
Let $\bm{\gamma} = (\gamma^{(1)}, ..., \gamma^{(s)})$,	
\begin{equation*}
\gamma^{(i)} = \sum_{r_i \in R_i} \frac{a_{r_i}}{b^{r_i}},
\end{equation*}
where
\begin{equation*}
R_i = \{s_i \in S_i| \; s_i \leq k\} \cup T_i, \text{ where } k := \Big[ \frac{m}{2(s-1)s} \Big],   
\end{equation*}
and where $t_i \in T_i$ has the form 
\begin{equation*}
t_i = \left[ \frac{m}{2s(s-1)} \right] + s j_i
\end{equation*}
for $i = 1,..., s-1$ and $t_s \in T_s$ has the form
\begin{equation*}
t_s = m - (s-1) \left(\left[\frac{m}{2s(s-1)}\right] + s \bar{m}\right) + s j_s.
\end{equation*}
Here, $j_1,...,j_s \in \{1,...,\bar{m}\}$ with 	
\begin{equation*}
\bar{m} = \left[\frac{m(2s-3)}{2s^2(s-1)}\right].
\end{equation*}
Moreover, we choose $a_{r_i} = a_{s_i}$ for all $r_i \in \{s_i \in S_i| \; s_i \leq k\}$ and otherwise, $a_{r_i} = 1$.\\
By the choice of $S_i$ it then holds that $[x^{(i)}]_k = [\gamma^{(i)}]_k$ for all $i = 1,...,s$. This implies that $\bm{x}$ and $\bm{\gamma}$ are contained in the same square elementary interval of order $s k$, i.e.,
\begin{equation*}
\bm{x}, \bm{\gamma} \in \prod_{i=1}^s \left[\frac{A_i}{b^k},\frac{A_i+1}{b^k}\right)
\end{equation*}
for some $A_i \in \{0,1,...,b^k-1\}$. Therefore, it holds that 
\begin{equation*}
\|\bm{x} - \bm{\gamma}\| < \sqrt{s}\frac{1}{b^k} \leq b\sqrt{s} \frac{1}{b^{ \frac{m}{2(s-1)s}}}.
\end{equation*}
Hence, (\ref{eq:gamma_first}) is shown.
It remains to check, whether the condition on $\bm{\gamma}$, mentioned at the beginning of the proof, are satisfied, i.e., if $\bm{\gamma} \in \Gamma$. Obviously, $R_i \subseteq \{1,2,...,m\}$ for all $i= 1,...,s$.

To begin with, observe that for any $r_i \in R_i$, where $i = 1,..., s-1$, and for any $s_s \in S_s, s_s \leq k$ we have that
\begin{align*}
\sum_{i=1}^{s-1} r_i + s_s 
&\leq (s-1)\left[\frac{m}{2s(s-1)}\right]+\bar{m}s + k \\
&\leq (s-1)\left(\frac{m}{2s(s-1)}\right)+\frac{m(2s-3)}{2s^2(s-1)}s +\frac{m}{2s(s-1)} \leq m.
\end{align*}
Additionally, for any $s_1 \in S_1, s_1 \leq k$ and $r_i \in R_i$, where $i = 2,...,s$ it holds that
\begin{align*}
s_1 + \sum_{i = 2}^s r_i 
&\leq k + (s-1)\left(\left[\frac{m}{2s(s-1)}\right] + s \bar{m}\right) +s \bar{m}\\
&\leq s \frac{m}{2s(s-1)} + (s-1) s  \frac{m(2s-3)}{2s^2(s-1)} + s  \frac{m(2s-3)}{2s^2(s-1)} = m.
\end{align*}
Hence, we can conclude that 
\begin{equation*}
|\{(r_1,...,r_s) | \; \sum_{i = 1}^s r_i > m , r_i \in R_i \}| = \abs{\{(t_1,...,t_s) | \; \sum_{i = 1}^s t_i > m , t_i \in T_i \}}.
\end{equation*}
Therefore, let us consider $t_i \in T_i$ for $i = 1,...,s$.	We have that 
\begin{equation*}
\sum_{i = 1}^s t_i = m + s(j_1 +...+j_s - (s-1) \bar{m}) \neq m+s,
\end{equation*}
because of the fact that $\bar{m} \in \Z$. It follows that
\begin{equation*}
|\{(r_1,...,r_s) | \; m+1 \leq \sum_{i = 1}^{s} r_i <  m + s \}| = 0.
\end{equation*}
For the case $t_1 + ...+t_s = m+s$ it holds that
\begin{equation*}
j_s = 1+(s-1) \bar{m} - j_1 - ... - j_{s-1}.
\end{equation*}
This implies that the following inequality must be fulfilled:
\begin{equation*}
1 \leq 1+(s-1) \bar{m} - j_1 - ... - j_{s-1} \leq \bar{m}.
\end{equation*}
Obviously, the left inequality holds for any choice of $j_1,...,j_{s-1}$. For the right inequality consider the case that $j_1 = ... = j_{s-1}$. Then we can conclude that it has to hold
\begin{equation*}
j_1 \geq \left[ \frac{(s-2)\bar{m}}{s-1} \right] +1.
\end{equation*}
Hence, we obtain
\begin{align*}
|\{(r_1,...,r_s) | \; \sum_{i = 1}^{s} r_i = m + s \}|\}| &= 
|\{(t_1,...,t_s) | \; \sum_{i = 1}^{s} t_i = m + s\}| \\
&= \left(\bar{m} - \left[ \frac{(s-2)\bar{m}}{s-1} \right]\right)^{s-1} \\
&\geq \left[ \frac{\bar{m}}{s-1} \right]^{s-1} \\
& \geq \frac{m^{s-1} (2s-3)^{s-1}}{(4s^2(s-1)^2)^{s-1}}
\end{align*}
by using the estimate 
\begin{equation*}
\left[\frac{\bar m}{s-1}\right] = \left[\frac{\left[\frac{m(2s-3)}{2s^2(s-1)}\right]}{s-1}\right]
\geq \frac{m(2s-3)}{4s^2(s-1)^2} \qquad \text{for } m \geq \frac{2s^2(s-1)^2}{2s-3}.
\end{equation*}
Thus, also (\ref{eq:gamma_second}) is shown. Now we finish the proof of Theorem \ref{thm:epsilon-umgebung-allgemein}. It remains to show the second item. Let $\mathcal{P} = \{\bm{x}_0, \dots, \bm{x}_{b^m-1}\}$ be a $(0,m,s)$-net in base $b$ for which some element $\bm{x}_i$ belongs to the set $\Gamma$. Therefore, the conditions of Lemma \ref{lma:levin5} are satisfied with $\alpha = s, \beta = \frac{(4s^2(s-1)^2)^{s-1}}{(2s-3)^{s-1}}$ and for any $\delta > \frac{b(b^{s-1}-1)(4s^2(s-1)^2)^{s-1}}{(2s-3)^{s-1}}$. By considering the limit $\delta \rightarrow \infty$ we obtain 
\begin{equation*}
	\frac{1}{N} \Delta(\bm{\gamma}, (\bm{x}_n)_{0 \leq n < b^m}) \leq - \frac{m^{s-1}}{b^m} \frac{(b-1)^s(2s-3)^{s-1}}{b^s(4s^2(s-1)^2)^{s-1}},
\end{equation*}
and the assertion follows with $N = b^m$.
\qed
\section{Re-proof of Theorem 2} 
In the interest of clear presentation, the proof of Theorem 2 will be split into several auxiliary lemmas. The necessity of the following two results should be motivated. In a later step, we will define a special axes-parallel box $[\boldsymbol{0}, \boldsymbol{y})$ and partition this multi-dimensional interval into several disjoint axes-parallel boxes (see, equation (\ref{eq:eq1})). Lemma 4.1 and Lemma 4.2 show under which condition on $n$ a sequence element $H_s(n)$ of the Halton sequence is contained in one of these disjoint intervals.   
\begin{lem}
Define $x_i:= \sum_{j=1}^{\infty} x_{i,j} b_i ^{-j}$, $x_{i,j} \in \lbrace 0, 1, \ldots, b_i -1 \rbrace$, and its truncation $[x_i]_r := \sum_{j=1}^r x_{i,j} b_i^{-j}$, for $i=1, \ldots, s$, $r=1,2, \ldots$. Then, we have 
\begin{align*}
\phi_{b_i}(n) \in [[x_i]_r,[x_i]_r + b_i^{-r}) \Longleftrightarrow  n \equiv \dot{x}_{i,r} \mod b_{i}^r,\text{ where } \dot{x}_{i,r} = \sum_{j=1}^r x_{i,j} b_i ^{j-1}.  
\end{align*}
\begin{proof}
The result follows immediately from the definition of the Halton sequence.
\end{proof}
\end{lem}
\begin{lem}
For a vector $\boldsymbol{r}=(r_1, \ldots, r_s)$ of positive integers, let $B_{\boldsymbol{r}}:= \prod_{i=1}^s b_i ^{r_i}$, and the integer $M_{i,\boldsymbol{r}}$, be defined such that $M_{i,\boldsymbol{r}} (B_{\boldsymbol{r}} b_i^{-r_i}) \equiv 1 \mod b_i^{r_i}$, then we have
\begin{equation*}
\phi_{b_i}(n) \in [[x_i]_{r_i},[x_i]_{r_i} + b_i^{-r_i}), \text{ for }i=1,\ldots, s \Longleftrightarrow n \equiv \ddot{x}_{{\boldsymbol{r}}} \mod B_{\boldsymbol{r}},
\end{equation*} with  $\ddot{x}_{{\boldsymbol{r}}}= \sum_{i=1}^{s} M_{i,\boldsymbol{r}} B_{\boldsymbol{r}} b_{i}^{-r_i} \dot{x}_{i,r_i}$.
\begin{proof}
This follows immediately from Lemma 4.1 and the Chinese remainder theorem. 
\end{proof}
\end{lem}
In order to obtain further information about the discrepancy function $\Delta(\cdot,(H_s(n))_{n=1}^{N})$ of the Halton sequence, we will investigate this function for a special setting of the interval $[\boldsymbol{0},\boldsymbol{y})$ and thereby exploit the information gained by the previous lemmas. Accordingly, let $y_i,\ i=1, \ldots, s$, be defined as 
\begin{equation*}
y_i:= \sum_{j=1}^m b_i^{-j \tau_i}, \text{ with } \tau_i= \min \lbrace 1 \leq k < B^{(i)} | b_i^k \equiv 1 \mod B^{(i)}  \rbrace, 
\end{equation*} where $m \in \mathbb{N}, m \geq B$ and $B^{(i)}= \frac{B}{b_i}$. If we consider, for instance, the two-dimensional Halton sequence in bases $b_1=2$ and $b_2=3$, we obtain $\tau_1=2$ and $\tau_2=1$. \\ \\
Having gathered these tools, we put $[\boldsymbol{0},\boldsymbol{y}) = [0,y^{(1)}) \times \ldots \times [0,y^{(s)}) \subset [0,1)^s$. The pertinence of introducing the integers $\tau_i$ will be revealed at a later step in Lemma 4.5. For a further analysis concerning $[\boldsymbol{0}, \boldsymbol{y})$, it turns out to be beneficial to consider a disjoint partitioning of this interval. To achieve the goal of a disjoint decomposition, a truncation of the one-dimensional interval borders $y_i$, of the form $ [y_i]_{\tau_i k_i} = \sum_{j=1}^{k_i} b_i^{-j\tau_i}$, $k_i \geq 1$, $i=1, \ldots, s$, is taken into account. Collecting the integers $k_i$ in a vector $\boldsymbol{k}= (k_1, \ldots, k_s)$ we arrive at 
\begin{equation}\label{eq:eq1}
[\boldsymbol{0}, \boldsymbol{y}) = \bigcup_{1 \leq k_1, \ldots, k_s \leq m} P_{\boldsymbol{k}}, \text{ with } P_{\boldsymbol{k}}: = \prod_{i=1}^s [[y_i]_{\tau_i k_i}-b_i^{-k_i \tau_i} ,[y_i]_{\tau_i k_i}).
\end{equation}  
We apply Lemma 4.2 to the interval $P_{\boldsymbol{k}}$ and obtain: 
\begin{lem}
An element $H_s(n)$ of the Halton sequence is contained in $P_{\boldsymbol{k}}$ if and only if $\phi_{b_i}(n) \in [[y_i]_{\tau_i k_i} - b_i^{-\tau_i k_i},[y_i]_{\tau_i k_i})$, for $i=1, \ldots,s$, or equivalently, 
\begin{equation}\label{eq:eq10} 
n \equiv \sum_{i=1}^s  M_{i,\boldsymbol{\tau} \cdot \boldsymbol{k}} B_{\boldsymbol{\tau} \cdot \boldsymbol{k}} b_i^{-\tau_i k_i} \dot{y}_{i,\tau_i (k_i-1)} \mod B_{\boldsymbol{\tau} \cdot \boldsymbol{k}}, \text{ where } \dot{y}_{i,\tau_i k_i}:=\sum_{j=1}^{k_i} b_i^{j\tau_i -1}. 
\end{equation} 
Note, that $\boldsymbol{\tau} = (\tau_1, \ldots, \tau_s)$ and the product $\boldsymbol{\tau} \cdot \boldsymbol{k}$ denotes the vector $(\tau_1 k_1, \ldots, \tau_s k_s)$.
\end{lem}
A slight reformulation of relation (\ref{eq:eq10}) is required. Although, by the previous lemma, we have found a criterion for a sequence element to be contained in $P_{\boldsymbol{k}}$, key steps of the proof of Theorem 2 will be based on a congruence of the form $n \equiv \tilde{y}_m + A_{\boldsymbol{k}} \mod B_{\boldsymbol{\tau} \cdot \boldsymbol{k}}$, with $\tilde{y}_m$ \textbf{independent} of $\boldsymbol{k}$ and $A_{\boldsymbol{k}}$ the least positive remainder modulo $B_{\boldsymbol{\tau} \cdot \boldsymbol{k}}$, i.e., 
\begin{equation*}
A_{\boldsymbol{k}}: \equiv  \sum_{i=1}^s -M_{i,\boldsymbol{\tau} \cdot \boldsymbol{k}} B_{\boldsymbol{\tau} \cdot \boldsymbol{k}} b_i^{-1} \mod  B_{\boldsymbol{\tau} \cdot \boldsymbol{k}} , \ A_{\boldsymbol{k}} \in [0, B_{\boldsymbol{\tau} \cdot \boldsymbol{k}}).
\end{equation*} 
This form is obtained as follows: We have 
\begin{align*}
& \sum_{i=1}^s  M_{i,\boldsymbol{\tau} \cdot \boldsymbol{k}} B_{\boldsymbol{\tau} \cdot \boldsymbol{k}} b_i^{-\tau_i k_i} \dot{y}_{i,\tau_i (k_i-1)} \\
&=\sum_{i=1}^s  M_{i,\boldsymbol{\tau} \cdot \boldsymbol{k}} B_{\boldsymbol{\tau} \cdot \boldsymbol{k}} b_i^{-\tau_i k_i} \dot{y}_{i,\tau_i k_i} - \sum_{i=1}^s M_{i,\boldsymbol{\tau} \cdot \boldsymbol{k}} B_{\boldsymbol{\tau} \cdot \boldsymbol{k}} b_i^{-1} \\
& \equiv \sum_{i=1}^s  M_{i,\boldsymbol{\tau} (m+1)} B_{\boldsymbol{\tau} (m+1)} b_i^{-\tau_i (m+1)} \dot{y}_{i,\boldsymbol{\tau}(m+1)} - \sum_{i=1}^s M_{i,\boldsymbol{\tau} \cdot \boldsymbol{k}} B_{\boldsymbol{\tau} \cdot \boldsymbol{k}} b_i^{-1} \\
& \equiv: \tilde{y}_m +  A_{\boldsymbol{k}} \mod B_{\boldsymbol{\tau} \cdot \boldsymbol{k}}. 
\end{align*}
Here $\tilde{y}_m$ is chosen such that $\tilde{y}_m \in [0, B_{\boldsymbol{\tau}(m+1)})$. The first of the congruences above follows by elementary computations. We summarize: 
\begin{equation*}
H_s(n) \in P_{\boldsymbol{k}} \Longleftrightarrow n \equiv \tilde{y}_m + A_{\boldsymbol{k}} \mod B_{\boldsymbol{\tau} \cdot \boldsymbol{k}}.
\end{equation*}  
Note that the multiplication $\boldsymbol{\tau} (m+1)$ has to be understood componentwise, i.e., we have $\boldsymbol{\tau} (m+1) = (\tau_1(m+1), \ldots, \tau_s(m+1))$. \\ \\
Employing the information received from Lemma 4.3, the equality   
\begin{equation*}
\sum_{n=N_1 B_{\boldsymbol{\tau} \cdot \boldsymbol{k}}}^{(N_1 +1) B_{\boldsymbol{\tau} \cdot \boldsymbol{k}} -1} (\chi_{P_{\boldsymbol{k}}}(H_s(n))- B_{\boldsymbol{\tau} \cdot \boldsymbol{k}}^{-1}) =0,
\end{equation*} holds for any integer $N_1 \geq 0$, since amongst $B_{\boldsymbol{\tau} \cdot \boldsymbol{k}}$ consecutive integers the congruence of relation (\ref{eq:eq10}) has exactly one solution. Moreover, for an integer $N_2 \in [0, B_{ \boldsymbol{\tau} \cdot \boldsymbol{k} })$, we have 
\begin{align}\label{eq:eq3}
 \sum_{n=\tilde{y}_m +N_1 B_{\boldsymbol{\tau} \cdot \boldsymbol{k} }}^{\tilde{y}_m +N_1 B_{\boldsymbol{\tau} \cdot \boldsymbol{k} }+ N_2 -1} (\chi_{P_{\boldsymbol{k}}}(H_s(n))- B_{\boldsymbol{\tau} \cdot \boldsymbol{k}}^{-1}) =\sum_{n \in [\tilde{y}_m,\tilde{y}_m +N_2)} (\chi_{P_{\boldsymbol{k}}}(H_s(n))- B_{\boldsymbol{\tau} \cdot \boldsymbol{k}}^{-1}).
\end{align} 
Recalling that 
\begin{align*}
& H_s(n) \in P_{\boldsymbol{k}} \Longleftrightarrow n \equiv \tilde{y}_m + A_{\boldsymbol{k}} \mod B_{\boldsymbol{\tau} \cdot \boldsymbol{k}} \Longleftrightarrow \\
&\exists \ l \in \mathbb{Z} \text{, such that } n = l B_{\boldsymbol{\tau} \cdot \boldsymbol{k}} + \tilde{y}_m + \underbrace{A_{\boldsymbol{k}}}_{\in [0,B_{\boldsymbol{\tau} \cdot \boldsymbol{k}})}, 
\end{align*}
the characteristic function in the sum (\ref{eq:eq3}) only has a nonzero contribution for $n=\tilde{y}_m + A_{\boldsymbol{k}}$, i.e., $l=0$, since for all other values of $l$, $n$ does not belong to the interval $[\tilde{y}_m,\tilde{y}_m +N_2)$. Hence, these arguments enable to restate (\ref{eq:eq3}) by the expression
\begin{align*}
\sum_{\substack{n \in [\tilde{y}_m,\tilde{y}_m +N_2) \\ n=\tilde{y}_m + A_{\boldsymbol{k}}}}  1 - N_2B_{\boldsymbol{\tau} \cdot \boldsymbol{k}}^{-1} &= 
\begin{cases}
1- N_2B_{\boldsymbol{\tau} \cdot \boldsymbol{k}}^{-1}, & \ 0 \leq  A_{\boldsymbol{k}} < N_2, \\
- N_2B_{\boldsymbol{\tau} \cdot \boldsymbol{k}}^{-1},  & \text{ else}.
\end{cases} \nonumber\\
&=\chi_{[0,N_2)}(A_{\boldsymbol{k}}) - N_2B_{\boldsymbol{\tau} \cdot \boldsymbol{k}}^{-1}.
\end{align*}
So far, we have constructed a special interval $[\boldsymbol{0},\boldsymbol{y})$, partitioned this box into subintervals and derived criteria to verify if some sequence element $H_s(n)$ is contained in a fixed box $P_{\boldsymbol{k}}$. To make the star-discrepancy of the Halton sequence sufficiently large, we additionally have to construct infinitely many values for $N$, which are bad in the sense that they yield (in combination with the special interval $[\boldsymbol{0},\boldsymbol{y})$) a large discrepancy. The decisive idea is to show the  existence of such $N$, rather to give an explicit construction. This consideration is realised by taking a quantity $\alpha_m$ into account, which represents the average of the discrepancy function, evaluated for the sequence elements $(H_s(n))_{ n=\tilde{y}_m}^{\tilde{y}_m + N -1}$ for several different values of $N$. Succeeding in showing that $\abs{\alpha_m} \geq c_s m^s$, with $c_s > 0$, would allow to conclude Theorem 2.      
\begin{lem}
Let 
\begin{equation*}
\alpha_m := \frac{1}{B_{\boldsymbol{\tau}m}} \sum_{N=1}^{B_{\boldsymbol{\tau}m}}  \Delta(\boldsymbol{y},(H_s(n))_{ n=\tilde{y}_m}^{\tilde{y}_m + N -1}),
\end{equation*} then 
\begin{equation}\label{eq:eq6}
\alpha_m = \sum_{1 \leq k_1, \ldots, k_s \leq m} \Big(\frac{1}{2} - \frac{A_{\boldsymbol{k}}}{B_{\boldsymbol{\tau} \cdot \boldsymbol{k}}}-\frac{1}{2 B_{\boldsymbol{\tau} \cdot \boldsymbol{k}}}\Big).
\end{equation}
\begin{proof}
We have
\begin{align*}
\alpha_m &= \frac{1}{B_{\boldsymbol{\tau}m}} \sum_{N=1}^{B_{\boldsymbol{\tau}m}} \Delta(\boldsymbol{y},(H_s(n))_{ n=\tilde{y}_m}^{\tilde{y}_m + N -1}) \\
 &= \sum_{1 \leq k_1, \ldots, k_s \leq m} \underbrace{\frac{1}{B_{\boldsymbol{\tau}m}} \sum_{N=1}^{B_{\boldsymbol{\tau}m}} \sum_{n=\tilde{y}_m}^{\tilde{y}_m + N -1} (\chi_{P_{\boldsymbol{k}}}(H_s(n))- B_{\boldsymbol{\tau} \cdot \boldsymbol{k}}^{-1})}_{=:\alpha_{m, \boldsymbol{k}}}. 
\end{align*}
The summands $\alpha_{m,\boldsymbol{k}}$ can be reformulated in the following way:
\allowdisplaybreaks 
\begin{align}\label{eq:eq4}
\alpha_{m,\boldsymbol{k}} &= \frac{1}{B_{\boldsymbol{\tau} m}} \sum_{N=1}^{B_{\boldsymbol{\tau} m}} \sum_{n=\tilde{y}_m}^{\tilde{y}_m + N -1} (\chi_{P_{\boldsymbol{k}}}(H_s(n))- B_{\boldsymbol{\tau} \cdot \boldsymbol{k}}^{-1}) \nonumber  \\
&=   \frac{1}{B_{\boldsymbol{\tau} m}} \sum_{N_1=0}^{B_{\boldsymbol{\tau} m}/B_{\boldsymbol{\tau} \cdot \boldsymbol{k}} -1} \sum_{N_2=1}^{B_{\boldsymbol{\tau} \cdot \boldsymbol{k}}}\Big( \underbrace{\sum_{n=\tilde{y}_m}^{\tilde{y}_m +N_1 B_{\boldsymbol{\tau} \cdot \boldsymbol{k} } -1} (\chi_{P_{\boldsymbol{k}}}(H_s(n))- B_{\boldsymbol{\tau} \cdot \boldsymbol{k}}^{-1})}_{=0} \nonumber \\
& + \underbrace{\sum_{n=\tilde{y}_m +N_1 B_{\boldsymbol{\tau} \cdot \boldsymbol{k} } }^{\tilde{y}_m +N_1 B_{\boldsymbol{\tau} \cdot \boldsymbol{k} }+ N_2 -1} (\chi_{P_{\boldsymbol{k}}}(H_s(n))- B_{\boldsymbol{\tau} \cdot \boldsymbol{k}}^{-1})}_{=\chi_{[0,N_2)}(A_{\boldsymbol{k}}) - N_2 B_{\boldsymbol{\tau} \cdot \boldsymbol{k}}^{-1}}\Big) \nonumber \\
&=  \frac{1}{B_{\boldsymbol{\tau} m}} \sum_{N_1=0}^{B_{\boldsymbol{\tau} m}/B_{\boldsymbol{\tau} \cdot \boldsymbol{k}} -1} \sum_{N_2=1}^{B_{\boldsymbol{\tau} \cdot \boldsymbol{k}}} (\chi_{[0,N_2)}(A_{\boldsymbol{k}}) - N_2 B_{\boldsymbol{\tau} \cdot \boldsymbol{k}}^{-1}) \nonumber  \\
&= \frac{1}{B_{\boldsymbol{\tau} \cdot \boldsymbol{k}}} \Big( \sum_{N_2=1}^{B_{\boldsymbol{\tau} \cdot \boldsymbol{k}}}\chi_{[0,N_2)}(A_{\boldsymbol{k}}) - \sum_{N_2=1}^{B_{\boldsymbol{\tau} \cdot \boldsymbol{k}}} N_2 B_{\boldsymbol{\tau} \cdot \boldsymbol{k}}^{-1}\Big).
\end{align}
By virtue of the fact that $A_{\boldsymbol{k}} \in [0,B_{\boldsymbol{\tau} \cdot \boldsymbol{k}})$ the first sum of (\ref{eq:eq4}) is not vanishing and simplifies to $B_{\boldsymbol{\tau} \cdot \boldsymbol{k}} - A_{\boldsymbol{k}}$. We therefore arrive at 
\begin{equation*}
\alpha_{m,\boldsymbol{k}} = \frac{1}{2} - \frac{A_{\boldsymbol{k}}}{B_{\boldsymbol{\tau} \cdot \boldsymbol{k}}}-\frac{1}{2 B_{\boldsymbol{\tau} \cdot \boldsymbol{k}}},
\end{equation*} and consequently
\begin{equation*}
\alpha_m = \sum_{1 \leq k_1, \ldots, k_s \leq m} \Big(\frac{1}{2} - \frac{A_{\boldsymbol{k}}}{B_{\boldsymbol{\tau} \cdot \boldsymbol{k}}}-\frac{1}{2 B_{\boldsymbol{\tau} \cdot \boldsymbol{k}}}\Big).
\end{equation*}
\end{proof}
\end{lem}
\begin{lem}
Let $\alpha_m$ be defined as in the previous lemma. Then we have 
\begin{equation*}
\abs{\alpha_m} \geq c_s m^s, \text{ with } c_s > 0. 
\end{equation*}
\begin{proof}
For simplicity reasons, we will prove this lemma only for the two-dimensional Halton sequence in bases $b_1=2$ and $b_2=3$. The general case works analogously with a bit more technical effort. To estimate the absolute value of $\alpha_m$ from below, we investigate the three occurring sums in (\ref{eq:eq6}) separately. We have $ \sum_{1 \leq k_1,k_2 \leq m} \frac{1}{2} = \frac{m^2}{2}$. The definition of $A_{\boldsymbol{k}}$ gives
\begin{equation}\label{eq:eq7}
\frac{A_{\boldsymbol{k}}}{B_{\boldsymbol{\tau} \cdot \boldsymbol{k}}} \equiv   - \sum_{i=1}^2 \frac{M_{i,\boldsymbol{\tau} \cdot \boldsymbol{k}} B_{\boldsymbol{\tau} \cdot \boldsymbol{k}} b_i^{-1}}{ B_{\boldsymbol{\tau} \cdot \boldsymbol{k}}}  \mod 1,
\end{equation}and therefore it is necessary to examine the expression $M_{i,\boldsymbol{\tau} \cdot \boldsymbol{k}} b_i^{-1} \mod 1$ in detail. According to the choice of the integer $M_{i,\boldsymbol{\tau} \cdot \boldsymbol{k}}$ and $\tau_i$, we obtain in our special case:
\begin{equation*}
M_{1,\boldsymbol{\tau} \cdot \boldsymbol{k}} 3^{k_2} \equiv 1 \mod 2^{2k_1},
\end{equation*}
hence
\begin{equation*}
M_{1,\boldsymbol{\tau} \cdot \boldsymbol{k}} 3^{k_2}  \equiv 1 \mod 2
\end{equation*}
and consequently 
\begin{equation*}
 M_{1,\boldsymbol{\tau} \cdot \boldsymbol{k}}  \equiv 1 \mod 2.
\end{equation*}
Further 
\begin{equation*}
M_{2,\boldsymbol{\tau} \cdot \boldsymbol{k}} 2^{2k_1} \equiv 1 \mod 3^{k_2},
\end{equation*}
hence
\begin{equation*}
M_{2,\boldsymbol{\tau} \cdot \boldsymbol{k}} 2^{2k_1}  \equiv 1 \mod 3
\end{equation*}
and consequently 
\begin{equation*}
 M_{2,\boldsymbol{\tau} \cdot \boldsymbol{k}}  \equiv 1 \mod 3.
\end{equation*}
Combining this result with~\eqref{eq:eq7} yields  
\begin{equation*}
\frac{A_{\boldsymbol{k}}}{B_{\boldsymbol{\tau} \cdot \boldsymbol{k}}} \equiv - \frac{1}{b_1} - \frac{1}{b_2} = -\frac{1}{2}-\frac{1}{3} \mod 1 = 1-\frac{1}{2}-\frac{1}{3} = \frac{1}{6}.
\end{equation*}
Summing up the reformulated addends of equation~\eqref{eq:eq6}, gives
\begin{align*}
\abs{\alpha_m} = \abs{ m^2 \Big(\frac{1}{2} - \frac{1}{6} \Big) - \sum_{1 \leq k_1,k_2 \leq m} \frac{1}{2 B_{\boldsymbol{\tau} \cdot \boldsymbol{k}}}} \geq c_2 m^2, \text{ with } c_2 > 0,
\end{align*} 
and $m$ sufficiently large.
\end{proof}
\end{lem}
This estimate gives us the necessary tools to conclude Theorem 2. \\ \\
\textbf{Proof of Theorem 2:} \\
From the definition of $\alpha_m$ (see formulation of Lemma 4.4) and from Lemma 4.5 we conclude that for every $m$ there is an $N$ with $1 \leq N \leq B_{\boldsymbol{\tau} m}$ such that 
\begin{equation*}
\abs{\Delta(\boldsymbol{y},(H_s(n))_{ n=\tilde{y}_m}^{\tilde{y}_m + N -1})} \geq c_s m^s.
\end{equation*}
Hence, 
\begin{equation*}
\abs{\Delta(\boldsymbol{y},(H_s(n))_{ n=0}^{\tilde{y}_m -1})} \geq \frac{c_s}{2} m^s \ \lor \ \abs{\Delta(\boldsymbol{y},(H_s(n))_{ n=0}^{\tilde{y}_m + N -1})} \geq \frac{c_s}{2} m^s.
\end{equation*}
Assume, the second estimate holds (the other case is treated analogously) and set $N_m:=\tilde{y}_m + N$, i.e., 
\begin{equation*}
\abs{\Delta(\boldsymbol{y},(H_s(n))_{n=0}^{N_m-1})} \geq \frac{c_s}{2} m^s.
\end{equation*}
Now note that 
\begin{equation*}
N_m = \tilde{y}_m + N \leq B_{\boldsymbol{\tau}(m+1)} + B_{\boldsymbol{\tau}m} \leq B^{3m(\tau_1 + \ldots + \tau_s)},
\end{equation*}
i.e.: 
\begin{equation*}
m \geq \frac{\log N_m}{\log B^{3(\tau_1 + \ldots + \tau_s)}},
\end{equation*}
and therefore 
\begin{equation*}
\abs{\Delta(\boldsymbol{y},(H_s(n))_{n=0}^{N_m-1})} \geq \frac{c_s}{2(\log B^{3(\tau_1+ \ldots + \tau_s)})^s} (\log N_m)^s.
\end{equation*}
It can easily be argued that we can obtain infinitely many such $N_m$, with this property and the result follows. \qed
\section{Proof of Theorem 5}
The investigations of the current section are restricted to the two-dimensional Halton sequence in bases $b_1=2$ and $b_2=3$. In the following, we survey possible options to modify the intervals $[0,y^{(1)})$ and $[0,y^{(2)})$, and discuss whether these changes still allow to derive the estimate $\abs{\alpha_m} \geq c_2 m^2$ or not. A way to obtain further possible values for $y^{(1)}$ or $y^{(2)}$ would be to remove some addends of the specification of $y^{(1)}$ or $y^{(2)}$, i.e., to consider for example 
\begin{equation*}
\tilde{y}^{(1)} = \sum_{\substack{j=1 \\ j \neq  l}}^{m} 2^{-j \tau_1}, \text{or } \tilde{y}^{(2)} = \sum_{\substack{j=1 \\ j \neq  l}}^{m} 3^{-j \tau_2}, \text{ with } l \in \mathbb{N} \text{ and } 1 \leq l \leq m. 
\end{equation*}
Recalling equation (\ref{eq:eq6}), the choice of the modified box $[0,\tilde{y}^{(1)}) \times [0,y^{(2)})$ would have the consequence that (\ref{eq:eq6}) amounts to 
\begin{equation*}
\alpha_m = \sum_{\substack{1 \leq k_1,k_2 \leq m \\ k_1 \neq l}} \Big(\frac{1}{2} - \frac{A_{\boldsymbol{k}}}{B_{\boldsymbol{\tau} \cdot \boldsymbol{k}}}-\frac{1}{2 B_{\boldsymbol{\tau} \cdot \boldsymbol{k}}}\Big).
\end{equation*} 
Note, that all previous steps of the proof of Theorem 2 can easily be adapted to this modified choice of the axes-parallel box. Since $k_1$ only takes on $(m-1)$ different values, we get
\begin{equation*}
\alpha_m = \frac{1}{3} m(m-1) - \sum_{\substack{1 \leq k_1,k_2 \leq m \\ k_1 \neq l}} \frac{1}{2 B_{\boldsymbol{\tau} \cdot \boldsymbol{k}}}
\end{equation*}
and therefore we are still in the position to derive a lower bound for $\abs{\alpha_m}$ of the form $c_2 m^2$. The next corollary focuses on the questions of how many addends can be removed from the representation of $y^{(1)}$ (or $y^{(2)}$).  
\begin{cor}
Let $\epsilon > 0$ and fix an $m > \hat{c}_2(\epsilon)$, with a sufficiently large constant $\hat{c}_2(\epsilon)$. If we remove at most $m(1-\epsilon)$ addends from the representation of $y^{(1)}$ ($y^{(2)}$), while $y^{(2)}$ ($y^{(1)}$) remains unchanged, then we still have $\abs{\alpha_m} \geq  c_2(\epsilon) m^2$, with $c_2(\epsilon)>0$.
\end{cor}
Up to now we have only modified $y^{(1)}$ ($y^{(2)}$) and kept $y^{(2)}$ ($y^{(1)}$) unchanged. If we remove addends from the representation of $y^{(1)}$ and from the one of $y^{(2)}$, we obtain the following corollary.
\begin{cor}
Let $\epsilon > 0$ and fix an $m > \hat{c}_3(\epsilon)$, with a sufficiently large constant $\hat{c}_3(\epsilon)$. If we remove at most $m(1-\epsilon)$ addends from the representation of $y^{(1)}$ and $y^{(2)}$ then we still have $\abs{\alpha_m} \geq  c_3(\epsilon) m^2$, with $c_3(\epsilon)>0$. 
\end{cor} 
Based on these preliminary considerations, we will derive the following lemma, which states, that there are, in some sense, many feasible choices for the interval borders $y^{(1)}$ and $y^{(2)}$.    
\begin{lem} 
Let $m$ be sufficiently large (as in Corollary 5.2). Then, there is a set $\Upsilon \subseteq [0,1)^2$ with the following property: 
For all $\boldsymbol{x} \in [0,1)^2$ there exists a $\boldsymbol{y} \in \Upsilon$ with 
\begin{equation*}
\| \boldsymbol{x} - \boldsymbol{y} \| < \sqrt{8} \frac{1}{2^{m/2}}.
\end{equation*}
Furthermore, for such a $\boldsymbol{y}$, we have $\abs{\alpha_m} \geq c_2 m^2$, with some constant $c_2 >0$.
\begin{proof}
Let $y^{(1)} = 0.\underbrace{010101 \ldots 01}_{2m}$ in base 2, and $y^{(2)}=0.\underbrace{11 \ldots 1}_{m}$ in base 3, the original choice of the interval borders of the two-dimensional box $[0,y^{(1)}) \times [0,y^{(2)})$. We now consider modified interval borders of the form $\tilde{y}^{(1)} = 0.\underbrace{a_1 \ldots a_{l_1} 0101 \ldots 01}_{2m}$, with $a_1, \ldots, a_{l_1} \in \lbrace 0,1 \rbrace$ and $\tilde{y}^{(2)} = 0.\underbrace{b_1 \ldots b_{l_2} 11 \ldots 11}_{m}$, with $b_1, \ldots, b_{l_2} \in \lbrace 0,1,2 \rbrace$. The question is of course, how large $l_1 = l_1(m)$ and $l_2 = l_2(m)$ can be chosen for a given $m$, such that we still have $\abs{\alpha_m} \geq c_2 m^2$ for this modified choice of the interval. The set $\Upsilon$ is then defined as the set of all feasible choices of $(\tilde{y}^{(1)},\tilde{y}^{(2)})$. Let $\tilde{k}_1^{(i)}$ and $\tilde{k}_1^{(i-1)}$ $\leq l_1/2$ be integers, for which $a_{2\tilde{k}_1^{(i)}} = a_{2\tilde{k}_1^{(i-1)}} = 1$. If one of the digits $a_{2 \tilde{k}_1^{(i-1)} + 1}, \ldots, a_{2 \tilde{k}_1^{(i)} -1}$ is one, we split an interval of the form 
\begin{equation*}
[[\tilde{y}^{(1)}]_{2\tilde{k}_1^{(i-1)}}, [\tilde{y}^{(1)}]_{2\tilde{k}_1^{(i)}})
\end{equation*}
into the two disjoint intervals
\begin{equation*}
[[\tilde{y}^{(1)}]_{2\tilde{k}_1^{(i-1)}}, [\tilde{y}^{(1)}]_{2\tilde{k}_1^{(i)}} - 2^{- 2\tilde{k}_1^{(i)}}) \ \land \ [ [\tilde{y}^{(1)}]_{2\tilde{k}_1^{(i)}} - 2^{-2\tilde{k}_1^{(i)}}, [\tilde{y}^{(1)}]_{2\tilde{k}_1^{(i)}}). 
\end{equation*}
Now, let $\tilde{k}_2^{(i)}$ $\leq l_2$, be an integer, for which $b_{\tilde{k}_2^{(i)}} = 2$.
Then, we split an interval of the form  
\begin{equation*}
[[\tilde{y}^{(2)}]_{\tilde{k}_2^{(i)}} - 2 \cdot 3^{-\tilde{k}_2^{(i)}}, [\tilde{y}^{(2)}]_{\tilde{k}_2^{(i)}})
\end{equation*} 
into the two disjoint intervals
\begin{equation*}
[[\tilde{y}^{(2)}]_{\tilde{k}_2^{(i)}} - 2 \cdot 3^{-\tilde{k}_2^{(i)}},[\tilde{y}^{(2)}]_{\tilde{k}_2^{(i)}} - 3^{-\tilde{k}_2^{(i)}}) \ \land \ [[\tilde{y}^{(2)}]_{\tilde{k}_2^{(i)}} - 3^{-\tilde{k}_2^{(i)}},[\tilde{y}^{(2)}]_{\tilde{k}_2^{(i)}}).
\end{equation*}
We investigate the influence of this additional interval on the quantity $\alpha_m$. Therefore, we consider the average of the discrepancy function for the interval 
\begin{equation*}
J_1 = [[\tilde{y}^{(1)}]_{2\tilde{k}_1^{(i-1)}}, [\tilde{y}^{(1)}]_{2\tilde{k}_1^{(i)}} - 2^{-2\tilde{k}_1^{(i)}}) \times [0,\tilde{y}^{(2)}),
\end{equation*}
\allowdisplaybreaks 
i.e., we study: 
\begin{align}
& \tilde{\alpha}^{(1)}_m = \frac{1}{B_{\boldsymbol{\tau}m}} \sum_{N=1}^{B_{\boldsymbol{\tau}m}} \Big( \sum_{n=\tilde{y}_m}^{\tilde{y}_m + N -1 } \chi_{J_1}(H_s(n)) - N \lambda_2(J_1)  \Big) \nonumber \\
& = \frac{1}{B_{\boldsymbol{\tau}m}} \sum_{N=1}^{B_{\boldsymbol{\tau}m}} \Big( \sum_{n=\tilde{y}_m}^{\tilde{y}_m + N -1 } \chi_{J_1}(H_s(n)) \Big) - \frac{B_{\boldsymbol{\tau} m}+1}{2} \Big( \sum_{j=2\tilde{k}_1^{(i-1)} +1 }^{2\tilde{k}_1^{(i)} - 1} \frac{a_j}{2^j} \Big( \sum_{i=1}^{l_2} \frac{b_i}{3^i} + \sum_{i=l_2+1}^{m} \frac{1}{3^i} \Big) \Big) \nonumber \\
& \geq \frac{1}{B_{\boldsymbol{\tau}m}} \sum_{N=1}^{B_{\boldsymbol{\tau}m}} \Big(\sum_{j=2\tilde{k}_1^{(i-1)} +1 }^{2\tilde{k}_1^{(i)} - 1} \sum_{i=1}^{l_2} a_jb_i \Big\lfloor \frac{N}{2^{j}3^i} \Big\rfloor + \sum_{j=2\tilde{k}_1^{(i-1)} +1 }^{2\tilde{k}_1^{(i)} - 1} \sum_{i=l_2+1}^{m} a_j\Big\lfloor \frac{N}{2^{j}3^i} \Big\rfloor  \Big) \nonumber \\
& - \frac{B_{\boldsymbol{\tau} m}+1}{2} \Big( \sum_{j=2\tilde{k}_1^{(i-1)} +1 }^{2\tilde{k}_1^{(i)} -1} \frac{a_j}{2^j} \Big( \sum_{i=1}^{l_2} \frac{b_i}{3^i} + \sum_{i=l_2+1}^{m} \frac{1}{3^i} \Big) \Big). \nonumber
\end{align}
Estimating the floor function yields: 
\begin{align*}
&  \tilde{\alpha}^{(1)}_m \geq \frac{1}{B_{\boldsymbol{\tau}m}} \sum_{N=1}^{B_{\boldsymbol{\tau}m}} \Big(\sum_{j=2\tilde{k}_1^{(i-1)} +1 }^{2\tilde{k}_1^{(i)} - 1}  \sum_{i=1}^{l_2} a_jb_i \Big( \frac{N}{2^{j}3^i} - 1 \Big) + \sum_{j=2\tilde{k}_1^{(i-1)} +1 }^{2\tilde{k}_1^{(i)} - 1} \sum_{i=l_2+1}^{m} a_j \Big( \frac{N}{2^{j}3^i} -1 \Big)  \Big) \\
& - \frac{B_{\boldsymbol{\tau} m}+1}{2} \Big( \sum_{j=2\tilde{k}_1^{(i-1)} +1 }^{2\tilde{k}_1^{(i)} -1} \frac{a_j}{2^j} \Big( \sum_{i=1}^{l_2} \frac{b_i}{3^i} + \sum_{i=l_2+1}^{m} \frac{1}{3^i} \Big) \Big) \\
& = \frac{B_{\boldsymbol{\tau} m}+1}{2} \sum_{j=2\tilde{k}_1^{(i-1)} +1 }^{2\tilde{k}_1^{(i)} - 1} \sum_{i=1}^{l_2} \frac{a_j}{2^{j}} \frac{b_i}{3^i} +  \frac{B_{\boldsymbol{\tau} m}+1}{2} \sum_{j=2\tilde{k}_1^{(i-1)} +1 }^{2\tilde{k}_1^{(i)} - 1}  \sum_{i=l_2+1}^{m} \frac{a_j}{2^{j}} \frac{1}{3^i} \\
& - \frac{B_{\boldsymbol{\tau} m}+1}{2} \Big( \sum_{j=2\tilde{k}_1^{(i-1)} +1 }^{2\tilde{k}_1^{(i)} -1} \frac{a_j}{2^j} \Big( \sum_{i=1}^{l_2} \frac{b_i}{3^i} + \sum_{i=l_2+1}^{m} \frac{1}{3^i} \Big) \Big) - \Big( \sum_{i=1}^{l_2} b_i + (m-l_2) \Big) \sum_{j=2\tilde{k}_1^{(i-1)} +1 }^{2\tilde{k}_1^{(i)} -1} a_j \\
& \geq (-m -l_2) \sum_{j=2\tilde{k}_1^{(i-1)} +1 }^{2\tilde{k}_1^{(i)} -1} a_j \\
& \geq -2m \sum_{j=2\tilde{k}_1^{(i-1)} +1 }^{2\tilde{k}_1^{(i)} -1} a_j. 
\end{align*}
We get an analogue upper bound for $\tilde{\alpha}^{(1)}_m$, by estimating $ \sum_{n=\tilde{y}_m}^{\tilde{y}_m + N -1 } \chi_{J_1}(H_s(n))$ with the expression  
\begin{equation*}
\sum_{j=2\tilde{k}_1^{(i-1)} +1 }^{2\tilde{k}_1^{(i)} - 1} \sum_{i=1}^{l_2} a_jb_i \Big( \Big\lfloor \frac{N}{2^{j}3^i} \Big\rfloor + 1 \Big) + \sum_{j=2\tilde{k}_1^{(i-1)} +1 }^{2\tilde{k}_1^{(i)} - 1} \sum_{i=l_2+1}^{m} a_j \Big( \Big\lfloor \frac{N}{2^{j}3^i} \Big\rfloor +1 \Big). 
\end{equation*}
To sum up, we get:
\begin{equation*}
\abs{\tilde{\alpha}^{(1)}_{m}} \leq 2m \sum_{j=2\tilde{k}_1^{(i-1)} +1 }^{2\tilde{k}_1^{(i)} -1} a_j.
\end{equation*}
In total, all intervals of this form yield therefore a contribution of at most $l_1m$. \\ \\ 
Studying the average of the discrepancy function for an interval of the form  
\begin{equation*}
J_2 = [0,\tilde{y}^{(1)}) \times [[\tilde{y}^{(2)}]_{\tilde{k}_2^{(i)}} - 3^{-\tilde{k}_2^{(i)}},[\tilde{y}^{(2)}]_{\tilde{k}_2^{(i)}}),
\end{equation*}
we get, analogously to above, an additional contribution to $\alpha_m$ of at most $l_2 m$. In total, we thus have, an contribution of the magnitude 
\begin{equation*}
m(l_1 + l_2).
\end{equation*}
Therefore, if $l_1 + l_2 < m$, we still can derive an estimate of the form $\abs{\alpha_m} \geq c_2 m^2$ for the modified box $[0, \tilde{y}^{(1)}) \times [0, \tilde{y}^{(2)})$. 
Let now $m$ be given and $\boldsymbol{x} =(x_{1}, x_{2}) \in [0,1)^2$, arbitrary but fixed, where 
\begin{equation*}
x_{1} = \sum_{i\geq 1} \frac{a_i}{2^i}, \ a_i \in \lbrace 0, 1 \rbrace \text{ and } x_{2} = \sum_{i\geq 1} \frac{b_i}{3^i}, \ b_i \in \lbrace 0, 1, 2 \rbrace.
\end{equation*}   
Due to above considerations, we can find $\boldsymbol{y} \in \Upsilon$, which satisfies    
\begin{equation*}
\| \boldsymbol{x} - \boldsymbol{y}  \| < \sqrt{ \Big( \frac{1}{2^{\lfloor \frac{m}{2} \rfloor -1}} \Big)^2 + \Big( \frac{2}{3^{\lfloor \frac{m}{2} \rfloor -1}} \Big)^2} < \sqrt{8} \frac{1}{2^{m/2}},
\end{equation*}
and also allows to derive $\abs{\alpha_m} \geq c_2 m^2$.
\end{proof}
\end{lem}
Based on the previous lemma, we are in the position to prove Theorem 5, which gives a lower bound for the discrepancy for a specific $N$ and not just for the average. \\\\
\textbf{Proof of Theorem 5:} \\
Fix an $m$, which satisfies the condition of Lemma 5.1 and recall $N_m = N + \tilde{y}_m$, as in the proof of Theorem 2. Consider now squares $Q_i \subseteq [0,1)^2$ of side length $\frac{2 \sqrt{8} }{2^{m/2}}$. Due to Lemma 5.1, we know that each such square contains elements of the set $\Upsilon$ (defined as in Lemma 5.1). We partition $[0,1)^2$ into $\frac{2^m}{32}$ such squares $Q_i$. Choose, for each $Q_i$, $\boldsymbol{y}_i \in Q_i \cap \Upsilon$. For some fixed $\boldsymbol{y}_i$, we have 
\begin{equation}\label{eq:eq100}
\abs{\alpha_m(\boldsymbol{y}_i)}  \geq c_2 m^2.
\end{equation}   
Let $c_2>0$ be small enough, such that this estimate holds for all other choices $\boldsymbol{y}_j \in Q_j \neq Q_i$ as well. \\
Note, that we always have $\abs{\alpha_m} \leq c m^2$ for a fixed constant $c>0$, since 
\begin{equation*}
D^{*}((H_2(n))_{n=1}^N) \leq c \frac{(\log N)^2}{N}, \text{ for all } N.
\end{equation*}
Now, we claim that the number of $N$s with $1 \leq N \leq B_{\boldsymbol{\tau}m}$ and 
\begin{equation*}
\abs{\Delta(\boldsymbol{y}_i, (H_2(n))_{n=1}^{N_m})} < \frac{c_2}{2} m^2
\end{equation*}
is at most $\kappa B_{\boldsymbol{\tau}m}$, with $\kappa :=\frac{c-c_2}{c - c_2/2}$. \\
Suppose the number of $N$s with $1 \leq N \leq B_{\boldsymbol{\tau}m}$ and 
\begin{equation*}
\abs{\Delta(\boldsymbol{y}_i, (H_2(n))_{n=1}^{N_m})} < \frac{c_2}{2} m^2
\end{equation*}
would be larger than $\kappa B_{\boldsymbol{\tau}m}$. Then, we would have 
\begin{equation*}
\abs{\alpha_m(\boldsymbol{y}_i) B_{\boldsymbol{\tau}m}} < \kappa B_{\boldsymbol{\tau}m} \frac{c_2}{2}m^2 + (1-\kappa)B_{\boldsymbol{\tau}m} c m^2 = c_2 B_{\boldsymbol{\tau}m} m^2, 
\end{equation*} 
which is a contradiction to inequality (\ref{eq:eq100}). \\ 
Therefore, the number of $N$s with $1 \leq N \leq B_{\boldsymbol{\tau}m}$ and 
\begin{equation*}
\abs{\Delta(\boldsymbol{y}_i, (H_2(n))_{n=1}^{N_m})} \geq \frac{c_2}{2} m^2
\end{equation*}
is at least $(1-\kappa)B_{\boldsymbol{\tau}m} = \frac{c_2}{2c - c_2} B_{\boldsymbol{\tau}m}$. \\ \\
To sum up, we have $\frac{2^m}{32}$ squares $Q_i$, and for each of them, we have identified $(1-\kappa)B_{\boldsymbol{\tau}m}$ distinct values for $N$, $1 \leq N \leq B_{\boldsymbol{\tau}m}$, which give a sufficiently large discrepancy. Thus, in total we have identified $\frac{2^m}{32}(1-\kappa)B_{\boldsymbol{\tau}m}$ many $N$ and this implies that at least one of those $N$ is identified at least $\frac{2^m}{32}(1-\kappa)$-times. Let $N_0$ be an $N$ with this certain multiplicity. Further, this means that there exist at least $\frac{2^m}{32}(1-\kappa)$ distinct $\boldsymbol{y}_i \in \cup_i \ Q_i \cap \Upsilon$, such that 
\begin{equation*}
\abs{\Delta(\boldsymbol{y}_i, (H_2(n))_{n=1}^{N_{m}^{(0)}})} \geq \frac{c_2}{2} m^2,
\end{equation*}
where $N_{m}^{(0)}:= N_0 + \tilde{y}_m$. Note, that the union of all squares $Q_i$ containing the $\boldsymbol{y}_i$ with this property, forms the set $\Lambda_{N_0}$ and therefore $\lambda_2(\Lambda_N) \geq 1 - \kappa$. It remains to verify, that for all $\boldsymbol{x} \in \Lambda_{N_0}$ there exists a $\boldsymbol{y} \in [0,1)^2$ having a distance less than $\sqrt{8} \frac{1}{N^{\frac{1}{14}}}$. Since $1 \leq N_0 \leq B_{\boldsymbol{\tau}m}$, the claim immediately follows by Lemma 5.1 and the estimate $\tilde{y}_m + B_{\boldsymbol{\tau}m} < 2^{7m}$. \qed
\begin{rem}
We note, that the considerations of this section can also be adopted to an arbitrary dimension $s>2$. For ease of notation, we have only presented them in the two-dimensional case for the bases $b_1=2$ and $b_2=3$. 
\end{rem} \noindent \\ 
\textbf{Acknowledgements.} We would like to thank our supervisor Gerhard Larcher for his valuable comments, suggestions and his general assistance during the writing of this paper.

\textbf{Author’s Addresses:} \\ 
Lisa Kaltenböck and Wolfgang Stockinger, Institut für Finanzmathematik und Angewandte Zahlentheorie, Johannes Kepler Universität Linz, Altenbergerstraße 69, A-4040 Linz, Austria. \\ \\ 
Email: lisa.kaltenboeck(at)jku.at, wolfgang.stockinger(at)jku.at 
\end{document}